\documentclass[11pt,english]{article}

\usepackage[margin = 2 cm, bottom=21 mm,footskip=7mm,top=20mm]{geometry}

\usepackage{amsthm}
\usepackage{amsmath}
\usepackage{amssymb}
\usepackage{setspace}
\usepackage{mathtools}
\usepackage{graphicx}
\usepackage[hidelinks]{hyperref}
\usepackage{thm-restate}
\usepackage{cleveref}
\usepackage{hyperref}
\usepackage{enumitem}
\usepackage{framed}
\usepackage{subcaption}

\usepackage[square,sort,comma,numbers]{natbib}
\usepackage{floatrow}

\renewenvironment{proof}[1][]{\begin{trivlist}
\item[\hspace{\labelsep}{\bf\noindent Proof#1.\/}] }{\qed\end{trivlist}}

\usepackage{tikz}
\usetikzlibrary{calc}
\usetikzlibrary{decorations.pathreplacing}
\usetikzlibrary{positioning,patterns}
\usetikzlibrary{arrows,shapes,positioning}
\usetikzlibrary{decorations.markings}

\def \bvx {circle[radius = .1][fill = black]}
\def \smvx {circle[radius = .07][fill = black]}

\tikzstyle{edge}=[very thick]
\definecolor{bostonuniversityred}{rgb}{0.8, 0.0, 0.0}
\definecolor{arsenic}{rgb}{0.23, 0.27, 0.29}
\tikzstyle{diredge}=[postaction={decorate,decoration={markings,
		mark=at position .55 with {\arrow[scale = 1.5]{stealth};}}}]
\tikzstyle{bidiredge}=[postaction={decorate,decoration={markings,
		mark=at position .7 with {\arrow[scale = 1.5]{stealth};}}}]

\tikzstyle{mdiredge}=[line width = 1 pt, postaction={decorate,decoration={markings,
		mark=at position .625 with {\arrow[scale = 1.5]{stealth};}}}]

\newcommand{\fitellipsis}[2] 
{\draw [fill=white]let \p1=(#1), \p2=(#2), \n1={atan2(\y2-\y1,\x2-\x1)}, \n2={veclen(\y2-\y1,\x2-\x1)}
    in ($ (\p1)!0.5!(\p2) $) ellipse [ x radius=\n2/2+0cm, y radius=0.25cm, rotate=\n1];
}

\newcommand{\remove}[1]{}

\newcommand{\floor}[1]{
    \lfloor #1 \rfloor
}

\newcommand{\eps}{\varepsilon}
\newcommand{\aad}{$1$-almost antidirected }

\usepackage{amsmath}
\usepackage{amssymb}
\usepackage{amsthm}
\usepackage{graphicx}
\usepackage{amscd}
\usepackage{epic, eepic}
\usepackage{url}
\usepackage{color}
\usepackage[utf8]{inputenc} 
\usepackage{enumerate}   
\usepackage{todonotes}
\usepackage{graphicx}
\usepackage{epstopdf}
\usepackage{enumitem}

\usepackage{comment}
\usetikzlibrary{arrows.meta}

\newtheorem{theorem}{\bf Theorem}[section]
\newtheorem{lemma}[theorem]{\bf Lemma}
\newtheorem{corollary}[theorem]{\bf Corollary}
\newtheorem{proposition}[theorem]{\bf Proposition}

\newtheorem{question}[theorem]{\bf Question}

 \theoremstyle{definition}

 \newtheorem*{definition*}{\bf Definition}

\def\ex{\mathrm{ex}}

\begin{document}

\title{\vspace{-0.9cm}Counting $H$-free orientations of graphs}
\date{}
\author{
Matija Buci\'c\thanks{Department of Mathematics, ETH Z\"urich, Switzerland. Email: \href{mailto:matija.bucic@math.ethz.ch} {\nolinkurl{matija.bucic@math.ethz.ch}}.}
\and
Oliver Janzer\thanks{Department of Mathematics, ETH Z\"urich, Switzerland. Email: \href{mailto:oliver.janzer@math.ethz.ch} {\nolinkurl{oliver.janzer@math.ethz.ch}}.
Research supported by an ETH Z\"urich Postdoctoral Fellowship 20-1 FEL-35.}
\and
Benny Sudakov\thanks{Department of Mathematics, ETH Z\"urich, Switzerland. Email:
\href{mailto:benjamin.sudakov@math.ethz.ch} {\nolinkurl{benjamin.sudakov@math.ethz.ch}}.
Research supported in part by SNSF grant 200021\_196965.}
}

\maketitle

\begin{abstract}
    In 1974, Erd\H os posed the following problem. Given an oriented graph $H$, determine or estimate the maximum possible number of $H$-free orientations of an $n$-vertex graph. When $H$ is a tournament, the answer was determined precisely for sufficiently large $n$ by Alon and Yuster. 
    In general, when the underlying undirected graph of $H$ contains a cycle, one can obtain accurate bounds by combining an observation of Kozma and Moran with celebrated results on the number of $F$-free graphs. As the main contribution of the paper, we resolve all remaining cases in an asymptotic sense, thereby giving a rather complete answer to Erd\H os's question. 
    Moreover, we determine the answer exactly when $H$ is an odd cycle and $n$ is sufficiently large, answering a question of Ara\'ujo, Botler and Mota.
\end{abstract}

\section{Introduction}

Given a fixed graph $H$, over all $n$-vertex graphs $G$ what is the maximum number of $2$-edge colourings of $G$ which contain no monochromatic copy of $H$? This very natural question was first asked by Erd\H{o}s and Rothschild \cite{E-R} in 1974 for the special case of $H=K_3$. This case was resolved by Yuster \cite{yuster} for large $n$ who in turn asked what happens for $H=K_k$. This problem, again for large $n$, was solved by Alon, Balogh, Keevash and Sudakov \cite{ABKS} who in addition solved it for $H$ being any edge-colour critical graph (defined as graphs in which the removal of some edge decreases the chromatic number). The question has attracted a lot of attention over the years and has been generalised in a number of ways; we point the interested reader to the numerous papers citing \cite{ABKS}, e.g. \cite{Bal06,LPRS09,PY12,PSY17}.

In the same paper from 1974, Erd\H os \cite{E-R} also raised the following closely related problem. 
Given an oriented graph $H$, what is the maximum possible number of $H$-free orientations of an $n$-vertex graph?
Let $D(n,H)$ denote the answer to this question. Erd\H os asked to determine or estimate $D(n,H)$. For an undirected graph~$F$, let $\ex(n,F)$ be the maximum number of edges in an $n$-vertex $F$-free graph. Writing $F$ for the underlying undirected graph of $H$, we have a trivial lower bound $D(n,H)\geq 2^{\ex(n,F)}$ since if $G$ is an $F$-free graph, then any orientation of $G$ is $H$-free. Alon and Yuster \cite{alon-yuster} showed that when $H$ is a tournament, this simple lower bound gives the correct answer. That is, if $T$ is a tournament on $k$ vertices, then $D(n,T)=2^{t_{k-1}(n)}$ holds for sufficiently large $n$, where $t_{k-1}(n)$ denotes the number of edges in the $(k-1)$-partite Tur\'an graph on $n$ vertices. Their general argument, which follows the approach used in \cite{ABKS}, relies on a regularity lemma and hence results in a requirement for $n$ to be extremely large. For the special case of $3$-vertex tournaments they give a different approach which solves the problem for the transitive tournament on three vertices for all $n$ and only requires $n$ to be larger than about $10000$ for the strongly connected 3-cycle $C_3$. As an aside, we remark that the number of $H$-free orientations of a random graph $G=G(n,p)$ has also been studied for various choices of $H$, e.g. for $H=C_k$ (see \cite{AKMP14,CKMM20}). Recently, Ara\'ujo, Botler and Mota \cite{mota} determined $D(n,C_3)$ for all values of $n$ and asked what happens if $H$ is an arbitrary strongly connected directed cycle, even if we are only interested in the case of large $n$. Our first result is an exact answer to their question for odd cycles. 

\begin{theorem}\label{thm:main}
    For any $k \ge 1$ there exists $n_0=n_0(k)$ such that if $n \ge n_0$, then $$D(n,C_{2k+1})=2^{\floor{n^2/4}}.$$
\end{theorem}

In fact, our argument, which follows closely the ideas of both \cite{ABKS} and \cite{alon-yuster}, applies for any $H$ which is an orientation of an edge-colour critical graph, showing that $D(n,H)=2^{\ex(n,F)}$ for large enough $n$ in such cases.

A natural next question is what happens for other graphs. As suggested by Erd\H{o}s, obtaining an approximate understanding of the answer is already interesting. Using a short and beautiful argument 
involving a version of the classical Sauer--Shelah lemma on VC dimension of sets, Kozma and Moran \cite{KM13} proved that the number of orientations of a 
\emph{fixed} graph $G$ without $H$ is always at most the number of $F$-free subgraphs of $G$, where as usual $F$ is the underlying graph of $H$. Hence, one can obtain upper bounds for $D(n,H)$ from known results on the number of $n$-vertex $F$-free graphs, which is an extensively studied subject on its own.

In the following result we trade precision for generality. It is obtained by combining the result of Kozma and Moran with that of Erd\H os, Frankl and R\"odl \cite{EFR86}.

\begin{proposition}\label{prop:general}
Given an oriented graph $H$ with underlying graph $F$, then 
$$D(n,H)=2^{\ex(n,F)+o(n^2)}.$$    
\end{proposition}

This result established the answer up to lower order terms for any oriented graph whose underlying graph is non-bipartite, since $\ex(n,F)=\Theta(n^2)$ for any non-bipartite $F$. The case of bipartite underlying graphs turns out to be more difficult, mostly due to the fact that their Tur\'an numbers are much less well understood. The next proposition relies on a result of Ferber, McKinley and Samotij \cite{FMS20}.

\begin{proposition} \label{prop:with cycle}
    Let $F$ be a graph containing a cycle, and assume that there exists a real number $\alpha$ such that $\ex(n,F)=\Theta(n^{\alpha})$. Then for any orientation $H$ of $F$, we have
    $$D(n,H)=2^{\Theta(n^{\alpha})}.$$
\end{proposition}

Although it is generally believed that such an $\alpha$ exists for every bipartite graph $F$, this is still a wide open conjecture. For a survey of the vast literature about the extremal number of bipartite graphs, see \cite{FS13}.

The above results provide us with good understanding of $D(n,H)$ whenever the underlying graph of $H$ contains a cycle. This leads to the natural question of what happens in the remaining case, namely when $H$ is an orientation of a forest $F$. Since in this case $\ex(n,F)=\Theta(n)$ (provided that $F$ has at least two edges), we have $D(n,H)=2^{\Omega(n)}$. On the other hand, the approach of bounding $D(n,H)$ with the number of $n$-vertex $F$-free graphs only gives $D(n,H)=2^{O(n\log n)}$. Up to this point, all the results are consistent with $D(n,H)=2^{\Theta(\ex(n,F))}$,  which might suggest that in the remaining cases, when $H$ is an orientation of a forest $F$, the same should hold.

As a natural starting point one might ask what happens with perhaps the simplest example of an oriented forest, namely the directed path on $k$ edges, which we denote by $P_k$. We show that in this case the trivial lower bound is indeed tight up to a multiplicative absolute constant in the exponent.

\begin{theorem}\label{thm:path}
    For any $k\geq 2$ and any $n\in \mathbb{N}$,
    $$D(n,P_k)\leq 2^{3kn}.$$
\end{theorem}

This result also suggests that $D(n,H)$ should always be $2^{\Theta(\ex(n,F))}$. However, perhaps surprisingly, it turns out that this is not the case. As we will see in a moment, there are even orientations of a path for which the answer is $2^{\Theta(n\log n)}$. We completely resolve the remaining cases by showing that for every oriented forest $H$ with at least two edges, either $D(n,H)=2^{\Theta(n)}$ or $D(n,H)=2^{\Theta(n\log n)}$. We also provide a precise characterisation for when each case occurs.

An oriented graph $H$ is said to be \emph{antidirected} if there exists a bipartition $V(H)=A\cup B$ of the vertex set such that every $u\in A$ has $0$ incoming edges and every $v\in B$ has $0$ outgoing edges. It is not too hard to see that antidirected forests are exactly those oriented forests $H$ with the property that any $n$-vertex directed graph with at least $Cn$ edges, for sufficiently large $C$, contains $H$. So in particular, $D(n,H)=2^{\Theta(n)}$ for any antidirected forest $H$. Despite this it turns out that there are many more oriented forests for which the answer is also $2^{\Theta(n)}$. The following definition precisely captures all such oriented forests.

\begin{definition*}
    We call an oriented graph $H$ \emph{1-almost antidirected} if there exists a bipartition $V(H)=A\cup B$ of the vertex set such that there are no edges inside $H[A]$ and $H[B]$, every $u\in A$ has at most one incoming edge and every $v\in B$ has at most one outgoing edge.
\end{definition*}

For example, the path with the usual orientation is $1$-almost antidirected, but there are orientations which are not. See Figure \ref{fig:path} for an orientation of the path with $5$ vertices which is not $1$-almost antidirected (the second and the fourth vertex must be on the same side in a bipartition, but the former has two out-edges and the latter has two in-edges).

We are now ready to state the value of $D(n,H)$ for oriented forests.

\begin{theorem}\label{thm:trees}
    Let $H$ be an oriented forest with at least two edges. Then 
    $$D(n,H)=
    \begin{cases}
    2^{\Theta(n)} & \text{if }H\text{ is }1\text{-almost antidirected}\\
    2^{\Theta(n\log n)} & \text{otherwise.}
    \end{cases}
    $$
\end{theorem}

\begin{figure}
	\centering
	\begin{tikzpicture}[scale=0.8]
	
	\node at (-1, 0) {$Q:$};
	\node (A) at (0, 0) {};
	\node (B) at (2, 0) {};
	\node (C) at (4, 0) {};
	\node (D) at (6, 0) {};
	\node (E) at (8,0) {};
	\node at (3.5, -1) {$D(n,Q)=2^{\Theta(n\log n)}$};

	\foreach \i in {A,B,...,E}
	{
	    \draw (\i) \smvx;
	}
	
	\draw[mdiredge, line width = 1pt] ($(B)$) -- ($(A)$);
	\draw[mdiredge, line width = 1pt] ($(B)$) -- ($(C)$);
	\draw[mdiredge, line width = 1pt] ($(C)$) -- ($(D)$);
	\draw[mdiredge, line width = 1pt] ($(E)$) -- ($(D)$);

    \node at (-13, 0) {$P_4:$};
    \node (a) at (-12, 0) {};
	\node (b) at (-10, 0) {};
	\node (c) at (-8, 0) {};
	\node (d) at (-6, 0) {};
	\node (e) at (-4,0) {};
	\node at (-8.5, -1) {$D(n,P_4)=2^{\Theta(n)}$};
	
	\foreach \i in {a,b,...,e}
	{
	    \draw (\i) \smvx;
	}
	
	\draw[mdiredge, line width = 1pt] ($(a)$) -- ($(b)$);
	\draw[mdiredge, line width = 1pt] ($(b)$) -- ($(c)$);
	\draw[mdiredge, line width = 1pt] ($(c)$) -- ($(d)$);
	\draw[mdiredge, line width = 1pt] ($(d)$) -- ($(e)$);

	\end{tikzpicture}
	\caption{$Q$ is an example of an oriented path which is not $1$-almost antidirected}
	\label{fig:path}
\end{figure}
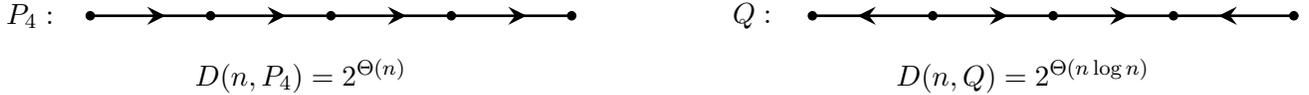

\textbf{Notation.} In this paper, no loops, multiple edges or bidirected edges are allowed in our oriented graphs. The underlying (undirected) graph of an oriented graph is the graph obtained by removing the orientations from all edges. An in-neighbour $u$ of a vertex $v$ is a vertex for which $uv$ is an edge. An in-leaf in an oriented tree is a leaf whose only edge is directed towards the leaf. We define out-neighbours and out-leaves analogously.

\section{Counting $H$-free orientations of graphs} 

\subsection{General oriented graphs} \label{sec:general}
In this section, we show how to deduce Proposition \ref{prop:general} and Proposition \ref{prop:with cycle} from known results.

Let $H$ be an arbitrary oriented graph. As before, let us write $F$ for the underlying undirected graph of $H$. For a graph $G$, write $D(G,H)$ for the number of $H$-free orientations of $G$ and denote by $N(G,F)$ the number of $F$-free subgraphs of $G$ on the same vertex set as $G$. Moreover, write $N(n,F)$ for the number of $F$-free graphs with vertex set $[n]$. Kozma and Moran \cite{KM13} proved that $N(G,F)$ is an upper bound for $D(G,H)$.

\begin{theorem}[Kozma--Moran \cite{KM13}] \label{thm:kozmamoran}
    Let $F$ be an undirected graph and let $H$ be an orientation of $F$. Then for any undirected graph $G$,
    $$D(G,H)\leq N(G,F).$$
    
    In particular, for any $n\in \mathbb{N}$,
    $$D(n,H)\leq N(n,F).$$
\end{theorem}

Perhaps surprisingly, the proof uses an inequality about set systems. Given a set system $\mathcal{A}$ on ground set $X$, we say that $S\subset X$ is \emph{shattered} by $\mathcal{A}$ if for every $T\subset S$, there exists some $A\in \mathcal{A}$ with $A\cap S=T$. Let us write $\mathrm{str}(\mathcal{A})$ for the collection of subsets of $X$ which are shattered by $\mathcal{A}$. A general version of the celebrated Sauer--Shelah lemma \cite{Paj85} states that $|\mathcal{A}|\leq |\mathrm{str}(\mathcal{A})|$.

To see that this inequality implies Theorem \ref{thm:kozmamoran}, fix an orientation $Q$ of $G$. Identify the power set of $E(G)$ with the set of orientations of $G$ by identifying $A\subset E(G)$ with the orientation of $G$ which differs from $Q$ precisely on the edge set $A$. Let $\mathcal{A}$ be the collection of subsets of $E(G)$ which are identified with $H$-free orientations of $G$. Clearly, $|\mathcal{A}|=D(G,H)$. On the other hand, assume that some $S\subset E(G)$ is shattered by $\mathcal{A}$. Then the graph formed by the edges in $S$ is $F$-free. Indeed, if it did contain $F$ as a subgraph, then there would exist an orientation of the edges in $S$ which contains a copy of $H$. This means that $S$ could not be shattered by $\mathcal{A}$. Thus, $|\mathrm{str}(\mathcal{A})|\leq N(G,F)$ and Theorem \ref{thm:kozmamoran} follows.

The function $N(n,F)$ has been extensively studied. Similarly to $D(n,H)$, we have the trivial lower bound $N(n,F)\geq 2^{\ex(n,F)}$ since if $G$ is an $F$-free graph, then any subgraph of $G$ is also $F$-free. This has been shown to be almost tight first for complete graphs by Erd\H os, Kleitman and Rothschild \cite{EKR76} and then for general non-bipartite graphs by Erd\H os, Frankl and R\"odl \cite{EFR86}, who proved that $N(n,F)=2^{\ex(n,F)+o(n^2)}$ (see \cite{balogh03} for an improved error term). Combined with Theorem \ref{thm:kozmamoran} and the trivial lower bound for $D(n,H)$, this implies Proposition \ref{prop:general}.

However, this gives an unsatisfactory answer for bipartite graphs $F$ as in that case $\ex(n,F)=O(n^{2-\eps_F})$ for some $\eps_F>0$. Since any $F$-free graph on vertex set $[n]$ has at most $\ex(n,F)$ edges, we get a straightforward upper bound
$$N(n,F)\leq \sum_{k=0}^{\ex(n,F)} \binom{\binom{n}{2}}{k},$$
which implies that $N(n,F)\leq 2^{C \ex(n,F)\log n}$ for some constant $C$. The logarithmic factor is necessary when $F$ is acyclic with maximum degree at least $2$. Indeed, in this case $\ex(n,F)=O(n)$, but it is easy to see that there are $2^{\Omega(n\log n)}$ graphs on vertex set $[n]$ with maximum degree $1$. On the other hand, it is not known in general whether the logarithmic factor is necessary for graphs that contain a cycle. In this direction, settling a classical conjecture of Erd\H{o}s it was shown by Morris and Saxton \cite{MS16} that $N(n,C_{2k})=2^{O_k(n^{1+1/k})}$, generalising a previous result of Kleitman and Winston \cite{K-W}, and complementing the classical Bondy--Simonovits bound $\ex(n,C_{2k})=O_k(n^{1+1/k})$ \cite{BS74}. Balogh and Samotij \cite{BS1,BS2} established a similar result for complete bipartite graphs in place of even cycles. These results were generalised greatly by Ferber, McKinley and Samotij \cite{FMS20}. They showed that if $F$ is a graph containing a cycle and there are positive constants $\alpha$ and $A$ such that $\ex(n,F)\leq An^{\alpha}$, then there exists a constant $C$ depending only on $\alpha$, $A$ and $F$ such that for all $n$,
    $N(n,F)\leq 2^{Cn^{\alpha}}$.
This result, combined with Theorem \ref{thm:kozmamoran} and $D(n,H)\geq 2^{\ex(n,F)}$, implies Proposition \ref{prop:with cycle}.

\subsection{Directed path}
In this subsection we will prove \Cref{thm:path}. Let $G$ be an $n$-vertex graph. Our task is to show that there are at most $2^{3kn}$ orientations of $G$ which do not contain $P_k$. Let us fix a canonical ordering of the vertices of $G$. We will count the number of orientations with the help of the following algorithm. It takes as input an orientation of $G$. In each step, it processes a vertex and updates the current ``state'' for every vertex that has not been processed yet. We will show that, provided the orientation is $P_k$-free, the potential states are severely restricted. We then use this to bound the number of possible orientations.

\textbf{Algorithm:} Initially, we assign to every vertex $v$ a state $(a_v,b_v)=(0,0)$. At step $i$ we have a sequence of already processed vertices $v_1,\ldots, v_{i-1}$ and possibly the next vertex to be processed $v_i$. 
\begin{itemize} 
\item If $v_i$ is not specified, it is chosen as a vertex in $V(G) \setminus \{v_1,\ldots, v_{i-1}\}$ with largest $a_{v_i}$, breaking ties by choosing such a $v_i$ first in the canonical ordering. 
\item We now proceed to process $v_i$. We consider all edges between $v_i$ and not already processed vertices.
\begin{itemize}
    \item If all these edges are oriented towards $v_i$, we do not specify $v_{i+1}$ and continue to the next step.
    \item Otherwise we choose $v_{i+1}$ to be the out-neighbour $v$ of $v_i$ with largest value of $b_{v}$ (breaking ties by choosing such a $v$ first in the canonical ordering) and
    \begin{itemize}
        \item we increase $a_u$ by one for any (non-processed) out-neighbour $u$ of $v_i$ and 
        \item we increase $b_u$ by one for any (non-processed) in-neighbour $u$ of $v_i$ which had $b_u\le b_{v_{i+1}}$. 
    \end{itemize} 
\end{itemize} 
\end{itemize}

Let us first motivate what is going on here. The algorithm reveals the orientation of the edges of $G$ bit by bit; specifically at step $i$ it will reveal the orientation of all (not already revealed) edges incident to $v_i$, the vertex we are currently processing. In particular, by the end of step $i$ the algorithm has revealed the orientation of all edges incident to $v_1,\ldots, v_i$. Note that the orientation of these edges determines uniquely the first $i$ steps of the algorithm, regardless of how the remaining edges are oriented (in other words, the algorithm is by this point completely independent of the orientation of the remaining edges). We roughly speaking think of the states $a_v$ and $b_v$ as the length of a directed path ending and starting at $v$, respectively, only using the already revealed edges, or in other words using only already processed vertices (apart from $v$ itself). This is captured more precisely in the following lemma. 

\begin{lemma}\label{lem:path-to-state}
    If the algorithm assigned state $(a_v,b_v)$ to $v$ (at any time) then the orientation contains a directed path of length $a_v$ ending in $v$ and a path of length $b_v-1$ starting with $v$. 
\end{lemma}
\begin{proof}    
    Let $v_1,\ldots, v_n$ be the order in which the vertices are processed. Observe that both $a_v$ and $b_v$ are non-decreasing throughout the process and are never updated once we process $v$. Let $a_t,b_t$ denote the final value of $a_{v_{t}},b_{v_t}$ so in particular at the point when we process $v_t$. 
    
    We are first going to show by induction on $t$ that there is a path of length at least $a_t$ ending in $v_t$ which only uses previously processed vertices. For the base case $t=1$, we know that $a_1=0$ so the claim is trivial. Now given $v_t$, let $i<t$ be the last index at which point $a_{v_t}$ was updated (if $a_{v_t}$ was never updated, then $a_t=0$ and our claim is trivial). Take the smallest $j$ such that $v_{j},v_{j+1},\ldots,v_i$ form a directed path oriented towards $v_i$. In particular, this means that $v_{j}$ was not an out-neighbour of $v_{j-1}$ and it was chosen as a vertex with maximum $a_{v}$ among all yet unprocessed vertices. As $v_t$ was not yet processed at this stage, it follows that we had $a_{v_t}\leq a_{v_j}$ at step $j$. Since after this point $a_{v_t}$ could only have been incremented when processing $v_j,\ldots, v_i$, we know that $a_t \le a_j+i-j+1$. In addition, we know by induction that there is a path of length $a_j$ ending at $v_j$ and using vertices only from $\{v_1,\dots,v_j\}$ to which we can append $v_{j}\ldots v_iv_t$ to obtain the desired path of length $a_j+i-j+1\ge a_t$.
    
    Turning now to the $b_v$'s, we are going to show by induction on $t$ that there is a path of length at least $b_t-1$ starting with $v_t$, which otherwise only uses vertices in $\{v_1,\ldots, v_{t-2}\}$. The base case $t=1$ is trivial. Also, observe that if $b_{v_t}$ got incremented at most once, i.e. $b_t\leq 1$, the claim also holds trivially. Let us now consider the penultimate (second to last) index $i$ at which point $b_{v_t}$ was updated. Observe that by definition, our algorithm will only increment $b_{v_t}$ when processing $v_j$ if $v_tv_j$ is an edge and $v_jv_{j+1}$ is an edge. In particular, this implies that we did not update $b_{v_t}$ at step $t-1$, and hence $i \le t-3$, since we chose $i$ to be the penultimate index which updated $b_{v_t}$. The fact that we incremented $b_{v_t}$ when processing $v_i$ means that at the time $b_{v_t}$ was at most $b_{v_{i+1}}$. Since this was the penultimate time $b_{v_t}$ was incremented, we know that $b_t \le b_{i+1}+2$. By induction we can find a path of length $b_{i+1}-1$ starting with $v_{i+1}$ which otherwise only uses vertices from the set $\{v_1,\dots,v_{i-1}\}$. This means that we can prepend $v_tv_iv_{i+1}$ to this path to obtain a new one of length $b_{i+1}-1+2 \ge b_t-1$ which only uses vertices from the set $\{v_1,\dots,v_{i+1}\}$ (apart from $v_t$). But $i\leq t-3$, so $\{v_1,\dots,v_{i+1}\}$ is a subset of $\{v_1,\dots,v_{t-2}\}$, as desired. 
\end{proof}

The above lemma tells us that if our orientation was $P_k$-free, then $a_v,b_v \le k$ throughout the process. Keeping this in mind, the following lemma considers the part of the orientation revealed before step $i$ of our algorithm and gives a bound on the number of ways in which one can complete this partial orientation into a $P_k$-free one. 

\begin{lemma}\label{lem:algorithm}
    Assume that up to step $i$ our algorithm processed vertices $v_1,\dots, v_{i-1}$ and assigned the state $(a_v,b_v)$ to any remaining vertex $v$. We fix the orientation of edges incident to vertices $v_1,\dots, v_{i-1}$ which leads to this state. There are at most $$\prod_{v \in V(G) \setminus \{v_1,\dots, v_{i-1}\}} (k+2)\cdot 2^{2k-a_v-b_v}$$ ways to orient the remaining edges to complete the orientation without creating a $P_k$.
\end{lemma}
Before turning to the proof, note that for $i=1$ no part of the orientation was specified and $a_v,b_v=0$ for every $v$, so the lemma tells us there are at most $(k+2)^n2^{2kn}\le 2^{3kn}$ orientations of $G$ without a $P_k$, establishing \Cref{thm:path}.
\begin{proof}
    Observe first that \Cref{lem:path-to-state} guarantees that $a_v,b_v \le k$ as otherwise any orientation we produce would contain a $P_k$.
    
    We will prove this by reverse induction on $i$. For the base case of $i=n$ there are no remaining edges to orient so the claim holds trivially, since $a_{v_n},b_{v_n} \le k$. 

    Let us assume that it holds if we start from step $i+1$ and any state. Given the orientation of the edges incident to $v_1,\ldots, v_{i-1}$, we can determine which vertex will be $v_i$. We will now consider all the ways in which we can extend our partial orientation to include the orientation of all edges incident to $v_i$. That is, we consider all the ways we can orient the edges between $v_i$ and $W:=V(G) \setminus \{v_1\ldots, v_{i}\}$. For each extended partial orientation, we will use the induction hypothesis to bound the number of ways it can be completed into a $P_k$-free orientation. Summing over all choices for the orientation of the edges between $v_i$ and $W$, we will get the desired bound. 
    
    If $v_i$ has all its unspecified edges (i.e. edges to $W$) oriented towards $v_i$ then the state of every remaining vertex remains unchanged and the induction hypothesis tells us that there are $\prod_{v \in W} (k+2)\cdot 2^{2k-a_v-b_v}$ ways to complete it into a $P_k$-free orientation.
    
    Otherwise our algorithm will choose $v_{i+1}$ to maximise $b_v$ among out-neighbours $v \in W$ of $v_i$. Assume this maximum is equal to $b$. Let us denote by $d$ the number of neighbours $v\in W$ of $v_i$ (in the underlying undirected graph $G$) which have $b_v$ at most $b$. We know that any other neighbour of $v_i$ in $W$ must be an in-neighbour of $v_i$ (by maximality of $b$). In particular, there are at most $2^{d}$ orientations of the edges incident to $v_i$ which result in this choice of $b$.
    For any such orientation, we claim that our algorithm increased the sum $\sum_{v \in W} (a_v+b_v)$ by $d$. Indeed, any out-neighbour of $v_i$ had its $a_v$ incremented and any in-neighbour among the $d$ neighbours with $b_v \le b$ had their $b_v$ incremented. This means that regardless of how we orient, by the induction hypothesis the number of ways to complete any of these orientations is at most $\frac{1}{2^d}\prod_{v \in W} (k+2)\cdot 2^{2k-a_v-b_v}.$ Since there are $2^d$ possible orientations of the edges incident to $v_i$ (given this choice of $b$) and $k+1$ choices for $0\le b \le k$, this gives at most 
    $(k+1)\prod_{v \in W} (k+2)\cdot 2^{2k-a_v-b_v}$ different ways to complete our initial orientation. Adding this to the contribution of the case without any out-neighbours of $v_i$ and using $a_{v_i},b_{v_i} \le k$ we obtain the claimed bound.
\end{proof}

\subsection{General trees}
In this section we will complete the proof of \Cref{thm:trees}. As already discussed in \Cref{sec:general}, we always have $D(n,H)\le 2^{O(n\log n)}$ when $H$ is an orientation of a forest.
Let us now take $G=K_{\lfloor n/2 \rfloor,\lfloor n/2 \rfloor}$ (with an extra isolated vertex if $n$ is odd), and orient all its edges from one part of the bipartition to the other (say from left to right), except for a single perfect matching which we orient in the other direction. Observe first that there are $\lfloor n/2 \rfloor !=2^{\Omega(n\log n)}$ such orientations, since this is the number of choices for the matching that we have. On the other hand, if we can embed an oriented forest $H$ in such an orientation, it must be $1$-almost antidirected, since the vertices of $H$ embedded in the left part of $G$ have in-degree at most $1$ (coming from the matching edges) and vertices embedded in the right part have out-degree at most $1$. This shows that if $H$ is not \aad then $D(n,H)=2^{\Theta(n\log n)}$, as claimed in the second part of \Cref{thm:trees}.

To establish the first part, we need to show that for any \aad forest $H$ there are at most $2^{O(n)}$ $H$-free orientations of any $n$-vertex graph. We will actually show that this holds for a certain universal oriented tree $H$ which contains all $k$-vertex \aad oriented forests. We define this universal oriented tree recursively, layer by layer. 

The starting point is the tree $H_{1,t}$, which is defined as follows. It has a root $v$ which has one in-neighbour $v_0$ and $t$ out-neighbours $v_1,\ldots, v_t$. Furthermore, $v_0$ has $t$ in-neighbours and each of $v_1,\ldots, v_t$ has a single out-neighbour and $t$ in-neighbours in addition to $v$. See \Cref{fig:H12} for an illustration. In our recursive definition it will be convenient to have another building block which we call $H_{1,t}^{-}$ and which is obtained from $H_{1,t}$ by deleting the subtree rooted at $v_0$. See \Cref{fig:Hminus12} for an illustration. $H_{s,t}$ is now defined by taking $H_{s-1,t}$ and appending a copy of $H_{1,t}$ to every out-leaf\footnote{Recall that an out-leaf is a leaf whose only edge is oriented away from the leaf.} and a copy of $H_{1,t}^{-}$ to every in-leaf. See \Cref{fig:H22} for an illustration.

Observe first that $H_{k,k}$ contains any \aad tree on $k$ vertices. This is due to the fact that in $H_{k,k}$ every non-leaf vertex at even depth has one in-neighbour and at least $k$ out-neighbours and every vertex at odd depth has one out-neighbour and at least $k$ in-neighbours, which allows one to simply greedily embed any \aad tree (starting from the root of $H_{k,k}$ to ensure that a leaf of $H_{k,k}$ is not encountered). Observe also that $H_{k+1,k}$ contains any \aad \emph{forest} on $k$ vertices. This follows since if we remove the root and its neighbours, the remaining oriented graph contains many (at least $k$) pairwise vertex-disjoint subtrees which are isomorphic to $H_{k,k}$, and we can embed the \aad trees making up our \aad forest into separate ones. With this in mind, the following theorem is the key result we need to prove in order to establish the remaining case of \Cref{thm:trees}.

\begin{figure}
\RawFloats
\begin{minipage}[t]{0.6\textwidth}
\centering
\captionsetup{width=\textwidth}

	\begin{tikzpicture}[scale=0.45]
	
	\node (A) at (0, 0) {$v$};
	\node (B) at (-8, -3) {$v_0$};
	\node (C) at (0, -3) {$v_{1}$};
	\node (D) at (8, -3) {$v_2$};
	
	\node (E) at (-10,-6) {};
	\node (F) at (-6,-6) {};
	\node (G) at (-2,-6) {};
	\node (H) at (0,-6) {};
	\node (I) at (2,-6) {};
	\node (J) at (6,-6) {};
	\node (K) at (8,-6) {};
	\node (L) at (10,-6) {};

    \draw[diredge] (B) -- (A);
    \draw[diredge] (A) -- (C);
    \draw[diredge] (A) -- (D);
    \draw[diredge] (E) -- (B);
    \draw[diredge] (F) -- (B); 
    \draw[diredge] (C) -- (G); 
    \draw[diredge] (H) -- (C); 
    \draw[diredge] (I) -- (C); 
    \draw[diredge] (D) -- (J);
    \draw[diredge] (K) -- (D);
    \draw[diredge] (L) -- (D);

	\end{tikzpicture}
	\caption{$H_{1,2}$}
	\label{fig:H12}
\end{minipage}\hfill
\begin{minipage}[t]{0.34\textwidth}
	\centering
	\begin{tikzpicture}[scale=0.45]
	
	\node (A) at (0, 0) {$v$};
	\node (C) at (0, -3) {$v_{1}$};
	\node (D) at (8, -3) {$v_2$};

		\node (G) at (-2,-6) {};
	\node (H) at (0,-6) {};
	\node (I) at (2,-6) {};
	\node (J) at (6,-6) {};
	\node (K) at (8,-6) {};
	\node (L) at (10,-6) {};
	
	\draw[diredge] (A) -- (C);
	\draw[diredge] (A) -- (D);
	\draw[diredge] (C) -- (G);
	\draw[diredge] (H) -- (C);
	\draw[diredge] (I) -- (C);
	\draw[diredge] (D) -- (J); 
	\draw[diredge] (K) -- (D);
	\draw[diredge] (L) -- (D);

	\end{tikzpicture}
	\caption{$H^-_{1,2}$}
	\label{fig:Hminus12}
\end{minipage}
\end{figure}

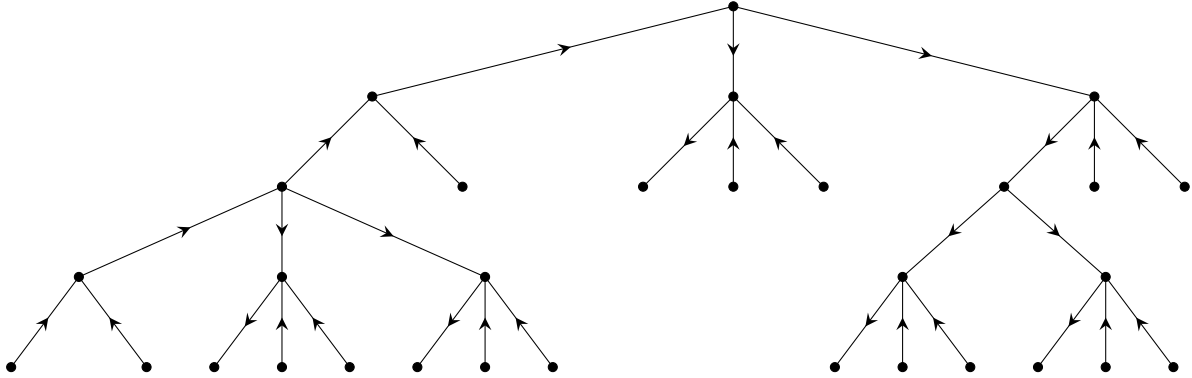
\begin{figure}
	\centering
	\begin{tikzpicture}[scale=0.6]
	
	\node (A) at (0, 0) {};
	\node (B) at (-8, -2) {};
	\node (C) at (0, -2) {};
	\node (D) at (8, -2) {};
	
	\node (E) at (-10,-4) {};

	\node (F) at (-6,-4) {};
	\node (G) at (-2,-4) {};
	\node (H) at (0,-4) {};
	\node (I) at (2,-4) {};
	\node (J) at (6,-4) {};
	\node (K) at (8,-4) {};
	\node (L) at (10,-4) {};
	
	\node (M) at (-14.5,-6) {};
	\node (N) at (-10,-6) {};
	\node (O) at (-5.5,-6) {};
	
	\node (P) at (3.75,-6) {};
	\node (Q) at (8.25,-6) {};
	
	\node (R) at (-16,-8) {};
	\node (S) at (-13,-8) {};
	
	\node (T) at (-11.5,-8) {};
	\node (U) at (-10,-8) {};
	\node (V) at (-8.5,-8) {};
	
	\node (W) at (-7,-8) {};
	\node (X) at (-5.5,-8) {};
	\node (Y) at (-4,-8) {};
	
	\node (t) at (2.25,-8) {};
	\node (u) at (3.75,-8) {};
	\node (v) at (5.25,-8) {};
	
	\node (w) at (6.75,-8) {};
	\node (x) at (8.25,-8) {};
	\node (y) at (9.75,-8) {};
	
	\foreach \i in {A,B,...,Y}
	{	
		\draw (\i) \bvx;
	}
	
	\foreach \i in {t,u,...,y}
	{	
		\draw (\i) \bvx;
	}
	
	\draw[diredge] ($(B)$) -- ($(A)$);
	\draw[diredge] ($(A)$) -- ($(C)$);
	\draw[diredge] ($(A)$) -- ($(D)$);
	\draw[diredge] ($(E)$) -- ($(B)$);
	\draw[diredge] ($(F)$) -- ($(B)$);
	\draw[diredge] ($(C)$) -- ($(G)$);
	\draw[diredge] ($(H)$) -- ($(C)$);
	\draw[diredge] ($(I)$) -- ($(C)$);
	\draw[diredge] ($(D)$) -- ($(J)$);
	\draw[diredge] ($(K)$) -- ($(D)$);
	\draw[diredge] ($(L)$) -- ($(D)$);
	\draw[diredge] ($(M)$) -- ($(E)$);
	\draw[diredge] ($(E)$) -- ($(N)$);
	\draw[diredge] ($(E)$) -- ($(O)$);
	\draw[diredge] ($(J)$) -- ($(P)$);
	\draw[diredge] ($(J)$) -- ($(Q)$);
	\draw[diredge] ($(R)$) -- ($(M)$);
	\draw[diredge] ($(S)$) -- ($(M)$);
	\draw[diredge] ($(O)$) -- ($(W)$);
	\draw[diredge] ($(X)$) -- ($(O)$);
	\draw[diredge] ($(Y)$) -- ($(O)$);
	\draw[diredge] ($(N)$) -- ($(T)$);
	\draw[diredge] ($(U)$) -- ($(N)$);
	\draw[diredge] ($(V)$) -- ($(N)$);
	\draw[diredge] ($(Q)$) -- ($(w)$);
	\draw[diredge] ($(x)$) -- ($(Q)$);
	\draw[diredge] ($(y)$) -- ($(Q)$);
	\draw[diredge] ($(P)$) -- ($(t)$);
	\draw[diredge] ($(u)$) -- ($(P)$);
	\draw[diredge] ($(v)$) -- ($(P)$);

	\end{tikzpicture}
	\caption{Part of $H_{2,2}$}
	\label{fig:H22}
\end{figure}

\begin{theorem} \label{thm:find Hst}
	Let $k,s,t$ be positive integers. Let $G=(X,Y)$ be a bipartite graph on $n$ vertices. Then there are at most $2^{O_{k,s,t}(n)}$ orientations of $G$ which contain no $H_{s,t}$ with the root in $X$ and in which there are at most $k$ out-edges from each $y\in Y$.
\end{theorem}

We will prove this theorem by induction on $s$. We will prove the case $s=1$ separately, as it will both serve as the base case and be useful in the induction step.

\begin{lemma} \label{lem:find H1t}
	Let $k,t$ be positive integers. Let $G=(X,Y)$ be a bipartite graph on $n$ vertices. Then there are at most $2^{O_{k,t}(n)}$ orientations of $G$ which contain no $H_{1,t}$ with the root in $X$ and in which there are at most $k$ out-edges from each $y\in Y$.
\end{lemma}

\begin{proof}
Set $T :=\max(|V(H_{1,t})|,kt)=O_{k,t}(1)$.

Let $D$ be an orientation of $G$ which contains no $H_{1,t}$ with the root in $X$ and in which there are at most $k$ out-edges from each $y\in Y$. Let $Y_1=Y$ and let $X_1$ be the set of vertices in $X$ which have at least one in-edge in $D$. Note that there are at most $2^n$ possibilities for $X_1$. Let $Y_2 \subseteq Y_1$ be the set of vertices in $Y_1$ which have at least $T+k$ neighbours in $X_1$. There are at most $2^{(T+k)(|Y_1|-|Y_2|)}$ ways to orient the edges incident to $Y_1 \setminus Y_2$. Consider the subset $X_2\subseteq X_1$ consisting of vertices which have an in-neighbour in $Y_2$. Since every vertex $y \in Y_1 \setminus Y_2$ had at most $T+k$ neighbours in $X_1$, there are at most $(T+k)(|Y_1|-|Y_2|)$ vertices in $X_1$ which have a neighbour in $Y_1 \setminus Y_2$, and we can specify the subset of them which got removed from $X_1$ to obtain $X_2$ in at most $2^{(T+k)(|Y_1|-|Y_2|)}$ many ways. We repeat as long as we can, i.e. until we obtain subsets $X_\ell\subset X_1$ and $Y_\ell\subset Y_1$ which have the following properties: every vertex of $X_\ell$ has an in-neighbour in $Y_\ell$ and every vertex in $Y_\ell$ has at least $T+k$ neighbours inside $X_\ell$. By the above counting, the number of possibilities for $X_\ell$, $Y_\ell$ and the orientations of the edges incident to $(X\setminus X_\ell)\cup (Y\setminus Y_\ell)$ is at most $2^n\cdot 2^{2(T+k)|Y_1|}\leq 2^{n+2(T+k)n}$.

If $Y_{\ell}=\emptyset$, then we have already revealed the entire orientation of $G$, so there are at most $2^{n+2(T+k)n}$ such suitable orientations. Assume that $Y_{\ell}\neq \emptyset$. We claim that this, together with the assumption that any vertex in $Y_\ell \subseteq Y$ has at most $k$ out-neighbours in $X$, guarantees that we can find a copy of $H_{1,t}$ in $D[X_\ell\cup Y_\ell]$ with the root in $X_\ell$, which is a contradiction. Indeed, fix one edge incoming from $Y_{\ell}$ at every vertex in $X_\ell$. These edges span vertex disjoint out-directed stars of size at most $k$ with centres in $Y_\ell$. In particular, there are at least $|X_{\ell}|/k$ centres. Since each centre has at least $T+k$ neighbours in $X_\ell$, at most $k$ of which can be out-neighbours, there are at least $T$ in-neighbours. Since $T \ge tk$, this means that some vertex $v \in X_\ell$ is an in-neighbour of at least $t$ distinct centres $v_1,\ldots, v_t$ of our out-stars. Picking one out-edge per star gives us an out-directed tree consisting of root $v$ and $t$ vertex-disjoint paths of length $2$. Note also that it is guaranteed that there is an in-neighbour $v_0\in Y_{\ell}$ of $v$ (which is distinct from $v_1,\ldots, v_t$ since they are its out-neighbours). What remains to be done is to find $t$ in-neighbours for each of $v_0,\ldots, v_t$ which we can do greedily since each of them has at least $T \ge |V(H_{1,t})|$ in-neighbours in $X_\ell$. So we found a copy of $H_{1,t}$ as claimed, and are done.
\end{proof}

Let us now define the oriented tree $H^*_{s,t,t'}$ by modifying $H_{s,t}$ so that every vertex in the penultimate layer has $t'$ instead of $t$ out-leaves attached to it. Note that if $t'\geq t$, then $H_{s,t'}$ contains $H^*_{s,t,t'}$ as a subgraph.

In the induction step, we will use the following key lemma.

\begin{lemma} \label{lem:extend H}
	For every $k,s,t$ there exists $t'=t'(k,s,t)$ as follows. Let $D=(X,Y)$ be a bipartite oriented graph such that there are at most $k$ out-edges from each $y\in Y$. Assume that $D$ contains a copy of $H^*_{s,t,t'}$ with the root in $X$ and assume that each leaf in this copy is the root of an $H_{1,t'}$. Then $D$ contains $H_{s+1,t}$ with the root in $X$.
\end{lemma}

\begin{proof}
	Let $t'$ be very large.
	
	Take a subgraph $K$ in $D$ which is isomorphic to $H^*_{s,t,t'}$ with the root in $X$ and in which every leaf is the root of an $H_{1,t'}$. Choose also a subgraph $L(w)$ isomorphic to $H_{1,t'}$ with root $w$ for every leaf $w$ in $K$. Observe that for every leaf $w$ of $K$, $w$ has a unique in-neighbour in $L(w)$, call this vertex $f(w)$ and note that $f(w)\in Y$.
	
	\medskip
	
	\noindent \emph{Claim.} $K$ has a subgraph $K'$ isomorphic to $H_{s,t}$ with the same root as $K$ such that for every out-leaf $w$ in $K'$, the vertex $f(w)$ is not in $V(K')$, and all these $f(w)$'s are distinct.
	
	\medskip
	
	\noindent \emph{Proof of Claim.} To get a subgraph of $H^*_{s,t,t'}$ isomorphic to $H_{s,t}$, we need to keep $t$ of the $t'$ out-leaves for every vertex of the penultimate layer. We can do this one by one in an arbitrary order. We just need to pay attention that for each out-leaf $w$ that we keep, the vertex $f(w)$ should be different from every vertex that is already in $K'$, and moreover all the $f(w)$'s for different $w$'s should be different. This can be done since $t'$ is very large, but every $f(w)$ is in $Y$ and every $y\in Y$ has at most $k$ out-edges in $D$, so at any point the number of forbidden choices for $w$ is at most $2|V(H_{s,t})|\cdot k$. In particular, $t'>2|V(H_{s,t})|\cdot k$ suffices. $\Box$
	
	\medskip
	
	It remains to extend $K'$ to a copy of $H_{s+1,t}$. For this, we need to ``attach" a copy of $H_{1,t}$ to each out-leaf in $K'$, and we need to attach a copy of $H_{1,t}^-$ to each in-leaf in $K'$ in a way that all new vertices are distinct from each other and from the vertices of $K'$. 
	
	We begin by extending $K'$ by joining $f(w)$ to $w$ for each out-leaf $w$ of $K'$. By the above claim, all of these $f(w)$'s are distinct and disjoint from $K'$. 
	Next for every leaf $w$ of $K'$ we want to append $t$ vertex disjoint (apart from sharing the start vertex $w$) paths of length $2$ directed away from $w$. This can be done since $L(w)$ gives us $t'>|V(H_{s+1,t})|$ vertex disjoint paths of length $2$ starting at $w$, so no matter how many vertices we already embedded we still have one of these paths available.
	
	It remains to attach $t$ in-neighbours to each vertex in the penultimate layer of our partial $H_{s+1,t}$. Since the vertices in the penultimate layer are in the middle layer of some $L(w)$, we know that each of these vertices has at least $t'$ in-neighbours in $D$ so once again we will always have enough of them to complete the picture.
\end{proof}

\begin{proof}[ of Theorem \ref{thm:find Hst}]
	We will use induction on $s$. The case $s=1$ is Lemma \ref{lem:find H1t}. Assume now that we have verified the statement for $s-1$. Let $k,s,t$ be positive integers and let $t'=t'(k,s-1,t)$ from Lemma~\ref{lem:extend H}.
	
	Let $D$ be an $H_{s,t}$-free orientation of $G$ with the property that there are at most $k$ out-edges from every $y\in Y$. Let $X_1$ be the set of vertices in $X$ which are roots of a copy of $H_{1,t'}$ in $D$ and let $X_2=X\setminus X_1$. Clearly there are at most $2^n$ possibilities for $X_1$. The oriented graph $D\lbrack X_2\cup Y\rbrack$ contains no $H_{1,t'}$ with the root in $X_2$, so the number of possibilities for the orientation of $G\lbrack X_2\cup Y\rbrack$ is at most $2^{O_{k,t'}(n)}$ by Lemma~\ref{lem:find H1t}. Moreover, since $D$ is $H_{s,t}$-free, Lemma~\ref{lem:extend H} implies that $D\lbrack X_1\cup Y\rbrack$ contains no copy of $H^*_{s-1,t,t'}$, and hence also no copy of $H_{s-1,t'}$, with the root in $X_1$. Then by induction there are at most $2^{O_{k,s-1,t'}(n)}$ possibilities for the orientation of $G\lbrack X_1\cup Y\rbrack$. Combining our bounds, the result follows.
\end{proof}

Finally let us deduce the first part of \Cref{thm:trees}.

\begin{corollary}
    Let $k \ge 1,$ let $G$ be an $n$-vertex graph and let $H$ be a $k$-vertex  \aad forest. There are at most $2^{O_k(n)}$ different  $H$-free orientations of $G$.
\end{corollary}

\begin{proof}
    The proof is by induction on $k$. For the base case note that the $k\le 2$ case is trivial. Let us now assume that the statement holds for forests with $k-1$ vertices. We may w.l.o.g.\ assume that $H$ has an in-leaf. Let $H'$ denote the oriented forest obtained from $H$ with this leaf removed. 
    Given an orientation $D$ of $G$, let $X$ be the subset of $V(G)$ consisting of vertices with out-degree at least $k$. There are $2^n$ different options for $X$. Let $Y=V(G) \setminus X$. Observe first that if we could find $H'$ inside $X$ then we could extend it to a copy of $H$ in $D$ since every vertex of $X$ has out-degree at least $k$. This means that by induction there are at most $2^{O_{k-1}(n)}$ many ways to orient the edges inside $X$. There are at most $|Y|k \le kn$ edges inside $Y$, so the edges inside $Y$ can be oriented in at most $2^{kn}$ many ways. Finally, the number of ways to orient the edges between $X,Y$ in a way that any vertex in $Y$ has at most $k$ out-edges and without creating a copy of $H_{k+1,k}$ is at most $2^{O_k(n)}$ by \Cref{thm:find Hst}. Since $H \subseteq H_{k+1,k}$, this completes the proof.
\end{proof}

\subsection{Odd cycles}

We will assume some familiarity with the basic directed regularity lemma, the specific details needed are given in Section 2 of \cite{alon-yuster}. Since the use of regularity in our argument is essentially the same as in both \cite{alon-yuster,ABKS}, we will not go into more technical details of these parts and will refer the reader to either of these papers, with the goal of making the key part of the argument easier to follow.

The following lemma says that if there are many orientations of $G$ which are $C_{2k+1}$-free then $G$ is not far from being bipartite. It is analogous to Lemma 2.1 in \cite{alon-yuster} which replaces $C_{2k+1}$ with an arbitrary tournament.

\begin{lemma}\label{lem:stability}
    Let $k \ge 1$ and $\delta>0$ there exists $n_0=n_0(\delta,k)$ such that if $G$ is a graph of order $n \ge n_0$ which has at least $2^{\floor{n^2/4}}$ distinct $C_{2k+1}$-free orientations then there is a bipartition of $V(G)$ with at most $\delta n^2$ edges inside parts.
\end{lemma}
    
\begin{proof}
Let us fix $0<\eps \ll \eta \ll \beta \ll \alpha \ll \delta$ as needed for various points of the upcoming argument. 

Let $\overrightarrow{G}$ be a $C_{2k+1}$-free orientation of $G$. We apply the directed regularity lemma to $\overrightarrow{G}$ to obtain an $\eps$-regular partition $V(\overrightarrow{G})=V_1 \cup \ldots \cup V_m$ (all $V_i$'s should have sizes as equal as possible, and all but $\eps m^2$ pairs $(V_i, V_j)$ should satisfy that linear sized subsets have about the same density of edges in both directions as the density between $V_i,V_j$). We then consider a cluster (di)graph $C$ of density $\eta$ (its vertices are the parts of our partition and two parts are joined by a directed edge if they are $\eps$-regular and the density of edges in the corresponding direction is at least $\eta$).

We first want to show that there exists some orientation $\overrightarrow{G}$ for which the resulting cluster graph has at least $m^2/4-\beta m^2$ edges directed both ways. We claim that if this is not the case, then there would be too few (less than $2^{\lfloor n^2/4\rfloor}$) orientations possible. 
Since the regularity lemma guarantees that $m \le M=M(\eps)$, there are at most $M^n$ choices for $\mathcal{P},$ at most $2^{\binom{M}{2}}$ choices for which pairs are $\eps$-regular and $4^{\binom{M}{2}}$ choices for $C$. In total there are at most $M^n2^{3M^2/2}$ choices for $\mathcal{P},$ regular pairs and $C$. Let us now bound how many orientations could give rise to a fixed choice. There are few edges inside parts of our fixed $\mathcal{P}$ and between non-$\eps$-regular pairs (at most $\eps n^2$ in both cases) and each edge may be oriented in two ways, so the total contribution of these edges to the number of orientations is at most a factor of $2^{2\eps n^2}$. For any $\eps$-regular pair $(V_i,V_j)$ which is not an edge of $C$ in one of the directions, there are at most about $\eta n^2/m^2$ directed edges in that direction. An easy estimate tells us that the edges between $V_i$ and $V_j$ can be oriented like this in at most $2^{c_{\eta}n^2/m^2}$ many ways where $c_{\eta} \to 0$ as $\eta \to 0$ and $c_{\eta}$ only depends on $\eta$. Since there are at most $m^2$ such pairs $(V_i,V_j)$, orienting edges between them contributes at most a factor of $2^{c_{\eta}n^2}$ to the total number of orientations. Finally, for any edge of $C$ directed both ways, there are at most $2^{(n/m)^2}$ orientations of the edges between the corresponding pair of parts, but since we are assuming that $C$ has at most $m^2/4-\beta m^2$ such edges, they contribute at most a factor of $2^{n^2/4-\beta n^2}$ to the total number of orientations. Putting it all together we get at most
$$M^n2^{3M^2/2}\cdot 2^{2\eps n^2} \cdot 2^{c_{\eta}n^2} \cdot 2^{n^2/4-\beta n^2}$$
orientations. Choosing $\eta$ to be small enough compared to $\beta$ gives us a contradiction to having at least $2^{\floor{n^2/4}}$ orientations.

Let now $\overrightarrow{G}$ be an orientation for which the resulting cluster graph $C$ has at least $m^2/4-\beta m^2$ edges directed both ways. We claim that $C$ can not contain a bidirected triangle missing only a single directed edge\footnote{In fact even having an oriented triangle with one bidirected edge would suffice, but this does not seem to be more useful.} as otherwise $\overrightarrow{G}$ would contain a $C_{2k+1}$. This is a consequence of a standard embedding lemma. One can easily deduce it from the classical (undirected) embedding lemma (see e.g. Lemma 2.1 in \cite{regularity}) by first refining the partition (splitting each part into $k$ parts of size as equal as possible, which preserves the regularity while density drops to at worst $\eta-\eps\ge \eta/2$) then only keeping edges in the desired direction, and applying the usual embedding lemma; see \Cref{fig:refinement} for an illustration.

\begin{figure}
    \centering
    \begin{tikzpicture}
        \node (a) at (-12,0) {};
	    \node (b) at (-11,1.73) {};
	    \node (c) at (-10,0) {};
	    \node at (-11,-1.4) {In cluster graph $C$};
	\foreach \i in {a,b,c}
	{	
		\draw (\i) \bvx;
	}
	\draw[bidiredge, line width = 1 pt] (a) -- (c);
	\draw[bidiredge, line width = 1 pt] (c) -- (a);
	\draw[diredge, line width = 1 pt] (a) -- (b);
	\draw[diredge, line width = 1 pt] (b) -- (c);

	\node (a1) at (-7,1) {};
	\node (a2) at (-6,-1) {};
	
	\node (b1) at (-4,2) {};
	\node (b2) at (-6,2) {};
	
	\node (c1) at (-4,-1) {};
	\node (c2) at (-3,1) {};
	\node at (-5,-1.4) {The refined cluster graph};

	\draw[] (-6.7,-0.1)--(-6.3,0.1);
	\draw[] (-3.7,0.1)--(-3.3,-0.1);
	
	\draw[bidiredge, line width = 1 pt] ($0.75*(a1)+0.25*(a2)$) -- ($0.75*(c1)+0.25*(c2)$);
	\draw[bidiredge, line width = 1 pt] ($0.75*(c1)+0.25*(c2)$) -- ($0.75*(a1)+0.25*(a2)$);
	\draw[bidiredge, line width = 1 pt] ($0.75*(c2)+0.25*(c1)$) -- ($0.75*(a1)+0.25*(a2)$);
	\draw[bidiredge, line width = 1 pt] ($0.75*(a1)+0.25*(a2)$) -- ($0.75*(c2)+0.25*(c1)$);
	\draw[bidiredge, line width = 1 pt] ($0.75*(c1)+0.25*(c2)$) -- ($0.75*(a2)+0.25*(a1)$);
	\draw[bidiredge, line width = 1 pt] ($0.75*(a2)+0.25*(a1)$) -- ($0.75*(c1)+0.25*(c2)$);
	\draw[bidiredge, line width = 1 pt] ($0.75*(a2)+0.25*(a1)$) -- ($0.75*(c2)+0.25*(c1)$);
	\draw[bidiredge, line width = 1 pt] ($0.75*(c2)+0.25*(c1)$) -- ($0.75*(a2)+0.25*(a1)$);
	
	\draw[bidiredge, line width = 1 pt] ($0.75*(a1)+0.25*(a2)$) -- ($0.5*(b1)+0.5*(b2)$);
	\draw[bidiredge, line width = 1 pt] ($0.75*(a2)+0.25*(a1)$) -- ($0.5*(b1)+0.5*(b2)$);
	
	\draw[line width = 1 pt, postaction={decorate,decoration={markings,
		mark=at position .4 with {\arrow[scale = 1.5]{stealth};}}}] ($0.5*(b1)+0.5*(b2)$) -- ($0.75*(c1)+0.25*(c2)$) ;
	\draw[line width = 1 pt, postaction={decorate,decoration={markings,
		mark=at position .45 with {\arrow[scale = 1.5]{stealth};}}}] ($0.5*(b1)+0.5*(b2)$) -- ($0.75*(c2)+0.25*(c1)$) ;
	
	\fitellipsis{a1}{a2};
	\fitellipsis{b1}{b2};
	\fitellipsis{c1}{c2};
	
	\draw[] (-6.7,-0.1)--(-6.3,0.1);
	\draw[] (-3.7,0.1)--(-3.3,-0.1);
	
		\node (a1) at (0,1) {};
	\node (a2) at (1,-1) {};
	
	\node (b1) at (3,2) {};
	\node (b2) at (1,2) {};
	
	\node (c1) at (3,-1) {};
	\node (c2) at (4,1) {};
	\node at (2,-1.4) {The $C_5$ we find.};
	
	\fitellipsis{a1}{a2};
	\fitellipsis{b1}{b2};
	\fitellipsis{c1}{c2};
	
	\draw[] (0.3,-0.1)--(0.7,0.1);
	\draw[] (3.3,0.1)--(3.7,-0.1);

	\draw[bidiredge, line width = 1 pt] ($0.75*(c1)+0.25*(c2)$) -- ($0.75*(a1)+0.25*(a2)$);

	\draw[diredge, line width = 1 pt] ($0.75*(a2)+0.25*(a1)$) -- ($0.75*(c1)+0.25*(c2)$);
	
	\draw[line width = 1 pt, postaction={decorate,decoration={markings,
		mark=at position .45 with {\arrow[scale = 1.5]{stealth};}}}] ($0.75*(c2)+0.25*(c1)$) -- ($0.75*(a2)+0.25*(a1)$);
	
	\draw[bidiredge, line width = 1 pt] ($0.75*(a1)+0.25*(a2)$) -- ($0.5*(b1)+0.5*(b2)$);

	\draw[line width = 1 pt, postaction={decorate,decoration={markings,
		mark=at position .45 with {\arrow[scale = 1.5]{stealth};}}}] ($0.5*(b1)+0.5*(b2)$) -- ($0.75*(c2)+0.25*(c1)$) ;
	
	\draw ($0.5*(b1)+0.5*(b2)$) \smvx;
	\draw ($0.75*(c2)+0.25*(c1)$) \smvx;
	\draw ($0.75*(a1)+0.25*(a2)$) \smvx;
	\draw ($0.25*(c2)+0.75*(c1)$) \smvx;
	\draw ($0.25*(a1)+0.75*(a2)$) \smvx;

    \end{tikzpicture}
    \caption{Illustration of how we find $C_5$.}
    \label{fig:refinement}
\end{figure}
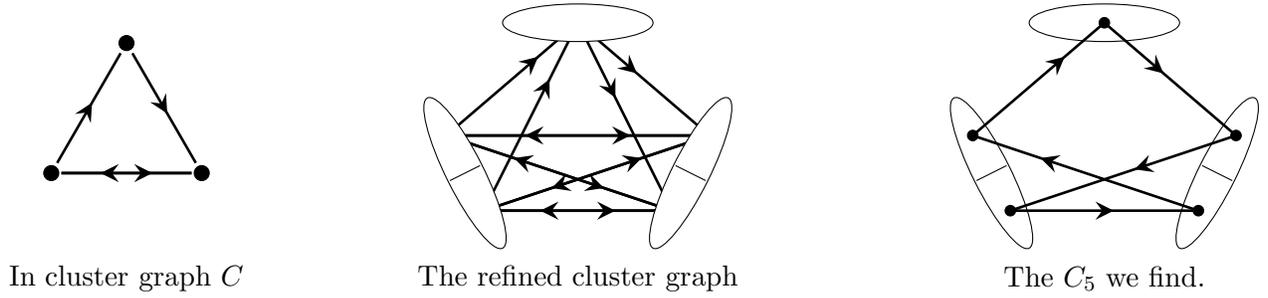

In particular, this tells us that the graph consisting only of the bidirected edges of $C$ is both triangle-free and has at least $m^2/4-\beta m^2$ edges. The stability theorem of Simonovits \cite{simonovits-stability} shows that there is a bipartition $V(C)=W_1 \cup W_2$ with at most $\alpha m^2$ bidirected edges within a part (using that $\alpha \gg \beta$). Hence, the bipartite subgraph consisting of the bidirected edges of $C$ between $W_1$ and $W_2$ has at least $m^2/4-(\beta+\alpha) m^2$ edges. If $C$ had in addition more than $8(\alpha+\beta)m^2$ directed edges inside parts, we would find a bidirected triangle with one directed edge removed in $C$. Indeed, more than $4(\alpha+\beta)m^2$ of these additional edges must be inside a single part, say $W_1$, and we can pass to a bipartite subgraph of size more than $2(\alpha+\beta)m^2$ within $W_1$. Taking into account these edges might also be bidirected there are more than $(\alpha+\beta)m^2$ \emph{distinct} pairs spanning a directed edge. These edges together with the bidirected edges between $W_1$ and $W_2$ make a subgraph of $C$ with more than $m^2/4$ edges, so by Mantel's theorem they give a triangle. This triangle has at most one edge inside $W_1$ (since the edges we used inside $W_1$ form a bipartite graph), so it has at least two bidirected edges, as desired.

It follows from the above that there are at most $\alpha m^2+8(\alpha+\beta)m^2$ edges of $C$ inside $W_1$ and $W_2$. Remove all edges of $G$ which correspond to such edges of $C$. Moreover, remove all edges within $V_i$'s and between pairs $(V_i,V_j)$ corresponding to non-edges in $C$. The remaining subgraph of $G$ is bipartite (with the parts being the union of $V_i$'s corresponding to $W_1$ and to $W_2$). Since there are at most $\eps n^2$ edges within $V_i$'s, at most $\eps n^2$ edges between non-$\eps$-regular pairs and at most $2\eta n^2$ edges between $\eps$-regular pairs which are non-edges in $C$, we have in total removed at most $(\alpha+8(\alpha+\beta))n^2+\eps n^2+\eps n^2+2\eta n^2\leq \delta n^2$ edges, as desired.
\end{proof}

The following lemma replaces the embedding Lemma 3.1 of \cite{alon-yuster}. Let us introduce some notation for convenience.
Given an oriented graph $D$ and an integer $k$, we say that a pair of disjoint subsets $W_1,W_2\subseteq V(D)$ with $|W_i|\ge 2k$ is $k$-regular if for any $X_i \subseteq W_i, |X_i|\ge |W_i|/20$, $D$ has at least $\frac{1}{10}|X_1||X_2|$ edges from $X_1$ to $X_2$, as well as at least $\frac{1}{10}|X_1||X_2|$ edges from $X_2$ to $X_1$.

\begin{lemma}\label{lem:embedding}
    Let $D$ be an oriented graph and let $W_1,W_2\subseteq V(D)$ be a $k$-regular pair. Then one can find a directed path of length $2k$ in the bipartite oriented graph $D[W_1,W_2]$ starting and ending in $W_1$.
\end{lemma}
    
\begin{proof}
    We iteratively find our directed path. Assume that for some $i\leq k$, we have found a path $v_1 v_2 \ldots v_{2i-1}$ and a subset $V_{2i-1}\subseteq W_2 \setminus \{v_2,v_4,\ldots,v_{2i-2}\}$ of at least $|W_2|/20$ out-neighbours of $v_{2i-1}$. Then since $D$ has at least $\frac{1}{10}|V_{2i-1}||W_1\setminus \{v_1,v_3,\ldots,v_{2i-1}\}|$ edges oriented from $V_{2i-1}$ to $W_1 \setminus \{v_1,v_3,\ldots,v_{2i-1}\}$, there must be a vertex $v_{2i}$ in $V_{2i-1}$ with a set $V_{2i}$ of at least $(|W_1|-k)/10\ge |W_1|/20$ out-neighbours in $W_1 \setminus\{v_1,v_3,\ldots,v_{2i-1}\}$. Repeating from the other side completes the iteration. After $k$ iterations, we find the desired path. 
\end{proof}

We now turn to the proof of the main result in this section.

\begin{proof}[ of \Cref{thm:main}]
    Let $n_0=n_0(\delta^2, k)$ be given by \Cref{lem:stability} applied with $\delta^2$ in place of $\delta$, for some sufficiently small $\delta$.
    
    Let us take a graph $G$ on $n>n_0^2+n_0$ vertices which has at least $2^{\floor{n^2/4}+m}$ $C_{2k+1}$-free orientations for some $m \ge 0$. We will show that if $G$ is not the Tur\'an graph, then either we can find a vertex $v$ such that $G\setminus v$  has at least $2^{\floor{(n-1)^2/4}+m+1}$ $C_{2k+1}$-free orientations or we can find distinct vertices $u$ and $v$ such that $G\setminus \{u,v\}$ has at least $2^{\floor{(n-2)^2/4}+m+2}$ $C_{2k+1}$-free orientations. We then iterate (note that no subgraph we consider can any longer be a Tur\'an graph since it has too many orientations, so also edges) as long as our graph has at least $n_0$ vertices. When we stop, we obtain a graph with less than $n_0$ vertices which has at least $2^{n_0^2}$ orientations, which is impossible, since it has at most $n_0^2/2$ edges.
    
    Let us assume that $G$ is not the Tur\'an graph on $n \ge n_0$ vertices and proceed to find a suitable vertex $v$.
    
    Every vertex needs to have degree at least $\floor{n/2}$ as otherwise its edges contribute at most a factor of $2^{\floor{n/2}-1}$ to the number of orientations so it would immediately work as our vertex $v$ above.
    
    Let $V_1,V_2$ form a bipartition of $V(G)$ which minimises the number of edges within parts. 
    Since $n \ge n_0$, by \Cref{lem:stability} we have at most $\delta^2 n^2$ edges within parts. This implies $|V_1|,|V_2| \le (1/2+\delta)n$ as otherwise $G$ would have less than $n^2/4$ edges, so too few orientations. Similarly, there can be at most $\delta^2 n^2$ edges missing between parts.
    
    We first claim that there can be only few orientations for which there exists a pair of disjoint subsets  $X_1\subseteq V_1, X_2\subseteq V_2$, both of size at least $2\delta n$, which have at most $|X_1||X_2|/10$ edges directed from $X_1$ to $X_2$. The number of such orientations of edges between $X_1,X_2$ is at most 
    $$\sum_{i=0}^{|X_1||X_2|/10}\binom{e(X_1,X_2)}{i} \le \sum_{i=0}^{|X_1||X_2|/10}\binom{|X_1||X_2|}{i}\le 2^{|X_1||X_2|/2},$$
    where $e(X_1,X_2)$ stands for the number of edges between $X_1$ and $X_2$. Since the total number of edges is at most $n^2/4+\delta^2 n^2$, there are at most $2^{n^2/4+\delta^2 n^2-|X_1||X_2|/2} \le 2^{n^2/4-\delta^2 n^2}$ such orientations of the whole graph.
    Since we can choose $X_1$ and $X_2$ in at most $2^{2n}$ many ways, there can be at most $2^{2n}\cdot 2^{n^2/4-\delta^2 n^2}\le 2^{\floor{n^2/4}-1}$ orientations for which such a pair $X_1,X_2$ exists.
        
    Let us now consider only $C_{2k+1}$-free orientations such that for any pair of disjoint subsets $X_1,X_2$ of size at least $2\delta n$, there are at least $|X_1||X_2|/10$ edges oriented from $X_1$ to $X_2$ and also from $X_2$ to $X_1$. Then any pair of subsets, both of size at least $40\delta n$, is $k$-regular. We call such an orientation \emph{relevant} and by the above counting and our assumption on the number of $C_{2k+1}$-free orientations of $G$, there are at least $2^{\floor{n^2/4}+m}-2^{\floor{n^2/4}-1}\ge 2^{\floor{n^2/4}+m-1}$ relevant orientations.
    
    \textbf{Case 1.} Some vertex $v$ has at least $800\delta n$ neighbours in its own part, say $V_1$.
    
    Note that $v$ must have at least $800 \delta n$ neighbours in $V_2$ as well, by maximality of the number of edges between $V_1$ and $V_2$. If in a relevant orientation $v$ has at least $40\delta n$ out-neighbours and at least $40\delta n$ in-neighbours belonging to different parts, then since these sets make a $k$-regular pair we can use Lemma~\ref{lem:embedding} to find a path of length $2k-1$ and join it with $v$ to give a $C_{2k+1}$, a contradiction. This means that at least $2$ out of the $4$ sets: out-neighbours of $v$ in $V_1$, in-neighbours of $v$ in $V_1$, out-neighbours of $v$ in $V_2$ and in-neighbours of $v$ in $V_2$ need to have size at most $40\delta n.$ These two parts can not belong to the same $V_i$ or be of different types in different parts. The only remaining option is for $v$ to have at most $80\delta n$ in-neighbours or at most $80\delta n$ out-neighbours. In particular, its edges may be oriented in such a way in at most 
    $$2 \sum_{i=0}^{80 \delta n} \binom{d(v)}{ i}\le 2 \sum_{i=0}^{d(v)/10} \binom{d(v)}{ i} \le 2^{0.49d(v)} \le 2^{0.49n}$$ 
    many ways, where $d(v)\le n$ denotes the degree of $v$ in $G$. In other words $G \setminus \{v\}$ must have at least $$2^{\floor{n^2/4}+m-1-0.49n}\ge 2^{\floor{(n-1)^2/4}+m+1}$$
    $C_{2k+1}$-free orientations, as desired.
    
    \textbf{Case 2.} Every vertex of $G$ has at most $800\delta n$ neighbours in its own part.
    
    Since $G$ is not the Tur\'an graph there must exist an edge $uv$ inside a part, say in $V_2$. Both $u$ and $v$ have at least $\floor{n/2}-800\delta n \ge n/3$ neighbours in $V_1$, in particular they have $d(u,v) \ge n/8$ common neighbours in $V_1$ since parts have size at most $n/2+\delta n$. If in a relevant orientation $uv$ is an edge, then the set $W_1$ of out-neighbours of $v$ in $V_1$ which are also in-neighbours of $u$ has size at most $40 \delta n$. This is due to \Cref{lem:embedding} (applied with $W_1$ and $W_2=V_2\setminus \{u,v\}$) which allows us to find a path of length $2k-2$ starting and ending in $W_1$, which in turn can be completed into a $C_{2k+1}$ using $uv$. This will severely reduce the number of possible orientations of edges incident to $u$ or $v.$ More precisely, the edges from $u$ and $v$ to their common neighbours can be oriented in at most 
    \vspace{-0.1cm}
    $$\sum_{i=0}^{40\delta n} \binom{d(u,v)}{i} \cdot 4^{i} \cdot 3^{d(u,v)-i}\le 41\delta n \cdot \binom{d(u,v)}{40\delta n} \cdot 4^{40\delta n} \cdot 3^{d(u,v)-40\delta n}\le  4^{0.99d(u,v)}.$$ 
    The same bound analogously holds if $vu$ is the edge instead. In particular, there are at most
    \vspace{-0.1cm}
    $$2^{d(u)+d(v)-0.02d(u,v)}\le 2^{n-n/1000}$$
    \vspace{-1cm}
    
    possible orientations of edges incident to $u$ or $v$ (we are using that both $u$ and $v$ have degree at most $n/2+ \delta n+800\delta n$ and that $\delta$ is small). In particular, the total number of orientations of $G \setminus \{u,v\}$ is at least 
    \vspace{-0.1cm}
    $$2^{\floor{n^2/4}+m-1}/2^{n-n/1000} \ge 2^{\floor{(n-2)^2/4}+m+2}$$
    \vspace{-0.3cm}and we are done.
\end{proof}
 
\section{Concluding remarks and open problems}

A classical result of Erd\H os and Gallai \cite{EG59} states that $\ex(n,P_k)\leq \frac{(k-1)n}{2}$ and this is tight when $k$ divides $n$. (With a slight abuse of notation, here $P_k$ referred to the unoriented path of length $k$.) The tightness of this bound implies that $D(n,P_k)\geq 2^{\frac{(k-1)n}{2}}$ when $k$ divides $n$, showing that our Theorem \ref{thm:path} is close to being optimal. Our bound coming from \Cref{lem:algorithm} is actually slightly better than $2^{3kn}$, namely $2^{(1+o_k(1))2kn}$, where $o_k(1)\rightarrow 0$ as $k\rightarrow \infty$. It seems not too hard to improve the constant further. We pose the following question.

\begin{question}
    Let $k$ be a positive integer. Does there exist $n_0=n_0(k)$ such that for every $n\geq n_0$, we have $D(n,P_k)\leq 2^{\frac{(k-1)n}{2}}$?
\end{question}

We were able to answer the question in the affirmative for $k\leq 3$. For general $k$, proving an upper bound of the form $2^{(1+o_k(1))kn/2}$ would already be interesting.

Turning to general oriented forests $H$, we have seen that $D(n,H)=2^{\Theta(\ex(n,F)\log n)}$ when $H$ is not $1$-almost antidirected (as before, $F$ is the underlying undirected graph). We also proved that the logarithmic factor is not present when $H$ is $1$-almost antidirected. However, even in that case, $D(n,H)$ need not be very close to $2^{\ex(n,F)}$ as the following example shows. Let $H$ be the oriented star on $k+1$ vertices whose edges are all oriented out from the centre. The underlying undirected graph has Tur\'an number about $\frac{(k-1)n}{2}$. On the other hand, we claim that for every $\eps>0$ there exists $k$ such that for all sufficiently large $n$, we have $D(n,H)\geq 2^{(1-\eps)k n}$. Indeed, given $\eps>0$, let $k$ be large and take a $(1-\frac{\eps}{2})2k$-regular graph $G_0$ on $2k$ vertices. If $k$ is large enough, using standard Chernoff estimates we can show that in a random orientation of $G_0$ the out-degree of every vertex is at most $(1-\frac{\eps}{4})k$ and in particular the orientation is $H$-free, with probability at least $1/2$. Hence, there are at least $\frac{1}{2}2^{(1-\frac{\eps}{2})2k^2}\geq 2^{(1-\eps)k|V(G_0)|}$ $H$-free orientations of $G_0$. So for large $n$ we can take $G$ to be the union of $\frac{n}{2k}$ vertex-disjoint copies of $G_0$, and then $D(G,H)\geq 2^{(1-\eps)kn}$.

However, it remains possible that there exists an absolute constant $C$ such that for every $1$-almost antidirected oriented tree $H$ there exists $n_0=n_0(H)$ such that for all $n\geq n_0$, we have $D(n,H)\leq 2^{C\ex(n,F)}$.

Turning to non-bipartite underlying graphs, Proposition \ref{prop:general} shows that $\log_2 D(n,H)$ is asymptotically $\ex(n,F)$ in this case. On the other hand, one can construct non-bipartite examples for which $D(n,H)>2^{\ex(n,F)}$ even for arbitrarily large $n$. For example, let $F$ be the complete bipartite graph $K_{5,5}$ with a path of length $4$ placed in one of the parts. Moreover, let $H$ be an orientation of $F$ in which every vertex has at least two out-neighbours and at least two in-neighbours in the other part and in which the path of length $4$ is oriented in a not \aad way (see Figure \ref{fig:path} for such an orientation of the path). Then $\ex(n,F)=\frac{n^2}{4}+O(n)$, but if $G$ is the complete tripartite graph with parts of size $n/2,n/4,n/4$, then $D(G,H)\geq 2^{n^2/4}\cdot (\frac{n}{4})!$, which shows that $D(n,H)>2^{\ex(n,F)}$.

One might also try to determine $D(n,H)$ more precisely when $H$ is a
forest which is not $1$-almost antidirected. Here it was observed by Alon
that if the
maximum degree of the forest $H$ is $d$, then $\log D(n,H)= O(n d \log n)$ (now the implied constant does not depend on $H$).
On the other hand, by an obvious modification of our lower bound
construction, one can show that there are oriented trees with maximum
degree $d$ such that $\log D(n,H)= \Omega(n d \log n)$. This determines
$\log D(n,H)$ up to an absolute constant factor for many, but not for all,
non $1$-almost antidirected forests~$H$.

\noindent \textbf{Acknowledgements.} We are grateful to Noga Alon and Rob Morris for helpful comments on a draft of this paper.

\bibliography{bib}
    \bibliographystyle{amsplain}

\end{document}